\documentclass[a4paper,11pt,draft]{article}
\usepackage{amsmath,amssymb,enumerate,color}
\usepackage{amsthm,color}
\usepackage{calc,txfonts}

\newcommand{\R}{\mathbb{R}}

\newcommand{\N}{\mathbb{N}}
\newcommand{\ep}{\varepsilon}
\newcommand{\pa}{\partial}

\DeclareMathOperator{\lifespan}{LifeSpan}

\newtheorem{theorem}{Theorem}[section]
\newtheorem{lemma}[theorem]{Lemma}
\newtheorem{proposition}[theorem]{Proposition}
\newtheorem{corollary}[theorem]{Corollary}
\theoremstyle{remark}
\newtheorem{remark}{Remark}[section]
\theoremstyle{definition}

\newtheorem{definition}{Definition}[section]
\newtheorem{example}{Example}
\numberwithin{equation}{section}

\usepackage{fancyhdr}

\begin{document}
\begin{center}
\Large{{\bf
Sharp lifespan estimates 
of blowup solutions to semilinear 
wave equations 
\\
with time-dependent effective damping
}}
\end{center}

\vspace{5pt}

\begin{center}
Masahiro Ikeda%
\footnote{
Department of Mathematics, Faculty of Science and Technology, Keio University, 3-14-1 Hiyoshi, Kohoku-ku, Yokohama, 223-8522, Japan/Center for Advanced Intelligence Project, RIKEN, Japan, 
E-mail:\ {\tt masahiro.ikeda@keio.jp/masahiro.ikeda@riken.jp}}, 
Motohiro Sobajima%
\footnote{
Department of Mathematics, 
Faculty of Science and Technology, Tokyo University of Science,  
2641 Yamazaki, Noda-shi, Chiba, 278-8510, Japan,  
E-mail:\ {\tt msobajima1984@gmail.com}}
and 
Yuta Wakasugi%
\footnote{
Graduate School of Science and Engineering, Ehime University, 
3, Bunkyo-cho, Matsuyama, Ehime, 790-8577, Japan, 
E-mail:\ {\tt wakasugi.yuta.vi@ehime-u.ac.jp}.}
\end{center}

\newenvironment{summary}{\vspace{.5\baselineskip}\begin{list}{}{%
     \setlength{\baselineskip}{0.85\baselineskip}
     \setlength{\topsep}{0pt}
     \setlength{\leftmargin}{12mm}
     \setlength{\rightmargin}{12mm}
     \setlength{\listparindent}{0mm}
     \setlength{\itemindent}{\listparindent}
     \setlength{\parsep}{0pt}
     \item\relax}}{\end{list}\vspace{.5\baselineskip}}
\begin{summary}
{\footnotesize {\bf Abstract.}
In this paper we consider the initial value problem for 
the semilinear wave equation with 
an effective damping
\begin{equation}\label{DW}
\tag*{(DW)}
\begin{cases}
\pa_t^2 u(x,t) -\Delta u(x,t) + b(t)\pa_t u(x,t)=|u(x,t)|^p, 
& (x,t)\in \R^N \times (0,T),
\\
u(x,0)=\ep f(x), 
& x\in \R^N,
\\ 
\pa_t u(x,0)=\ep g(x),  
& x\in \R^N,
\end{cases}
\end{equation}
where $N\in \N$, $0<b\in C^1([0,\infty))$, $1<p\leq 1+\frac{2}{N}$
and $\ep>0$ is a parameter 
describing the smallness of initial data.  
Here the coefficient $b(t)$ of the damping term 
is assumed to be ``effective''.
The interest is the behavior of lifespan of solutions in view of 
the asymptotic profile of $b(t)$ as $t\to\infty$. 
The simple cases $b(t)=(1+t)^{-\beta}$ ($\beta\in (-1,1)$) 
and the threshold case $b(t)=1+t$ in the sense of overdamping 
are discussed in Ikeda--Wakasugi \cite{IkeWa15} and Ikeda--Inui \cite{IkeInpre}, respectively. 
In the present paper we discuss general damping terms
with a certain assumption. 
The result of this paper is the sharp lifespan estimates 
of blowup solutions to \ref{DW} including the typical case $b(t)=(1+t)(1+\log(1+t))$. 
The proof of upper bound of lifespan 
is a modification of the test function method given in \cite{IkedaSobajima3} 
and the one of lower bound is 
based on the technique of scaling variables 
introduced in Gallay--Raugel \cite{GaRa98} (for $N=1$) 
and Wakasugi \cite{Wakasugi17} (for $N\geq 2$).

}
\end{summary}



\section{Introduction}
In this paper we consider the blowup phenomena for 
the initial value problem of the semilinear 
wave equation with an effective damping of the form 
\begin{equation}\label{ndw}
\begin{cases}
\pa_t^2 u(x,t) -\Delta u(x,t) + b(t)\pa_t u(x,t)=|u(x,t)|^p, 
& (x,t)\in \R^N \times (0,T),
\\
u(x,0)=\ep f(x), 
& x\in \R^N,
\\ 
\pa_t u(x,0)=\ep g(x),  
& x\in \R^N,
\end{cases}
\end{equation}
where $N\in\N$,
$b\in C^1([0,\infty))$, 
$\ep>0$ is a small parameter
and 
$f,g$ 
are given functions satisfying 
\[
(f,g)\in (H^1(\R^N)\cap L^1(\R^N))\times (L^2(\R^N)\cap L^1(\R^N)).
\]
The term $b(t)\pa_t u$ 
describes the damping effect 
which provides the reduction of the energy as a wave.
Therefore the size of $b(t)$ could affect 
to the profile of solution for sufficiently large $t$.  
The interest of this paper is to 
clarify the effect of damping 
coefficient in terms of 
the behavior of the lifespan with respect to $\ep$.

The equation in \eqref{ndw} with $b(t)\equiv 1$ (without a nonlinear term)
was introduced in Cattaneo \cite{Ca58} and Vernotte \cite{Ve58} to consider 
a model of heat conduction with finite propagation property. 
This equation is composed by ``balance law'' $u_t={\rm div}\,q$ 
and ``time-delayed Fourier law'' $\tau q_t+q=\nabla u$, where $q$ is the heat flux
and $\tau$ is sufficiently small. 

In the case $b(t)\equiv 1$, the equation \eqref{ndw} becomes 
the usual damped wave equation and therefore
there are many 
previous works dealing with 
global existence 
and blowup of solutions to \eqref{ndw} with lifespan estimates
(see e.g., 
Li--Zhou \cite{LiZh95}, Todorova--Yordanov \cite{ToYo01}, Nishihara \cite{Ni03Ib},
Ikeda--Wakasugi \cite{IkeWa15},
Ikeda--Ogawa \cite{IO16}, Lai--Zhou \cite{LaZhpre}).
As a summary, 
the Fujita exponent $p=1+2/N$ plays a role of critical exponent 
dividing the global existence and blowup of small solutions. The lifespan estimates 
are given as the following: 
\[
\lifespan(u)\sim 
\begin{cases}
C\ep^{-(\frac{1}{p-1}-\frac{N}{2})^{-1}}
&
{\rm if\ }1<p<1+\frac{2}{N},
\\
\exp(C\ep^{-(p-1)})
&
{\rm if\ }p=1+\frac{2}{N},
\\
\infty
&
{\rm if\ }p>1+\frac{2}{N}
\end{cases}
\]
for sufficiently small $\ep>0$.

In the case $b(t)=(1+t)^{-\beta}$ with $\beta\in (-1,1)$, 
Lin--Nishihara--Zhai \cite{LNZ12} found that 
the critical exponent in this case remains $p=1+2/N$. 
Later, the damping is generalized to the profile of $b(t)$ as $t\to\infty$ 
by D'Abbicco--Lucente \cite{DL13} and D'Abbicco--Lucente--Reissig \cite{DLR13} 
and then the critical exponent remains $p=1+2/N$ again. 

We have to mention that 
$b(t)=\frac{\mu}{1+t}$ is so-called scale-invariant 
damping and in this case 
the effect of wave structure cannot be ignored
in the sense of existence of global solutions.
Actually, in Ikeda--Sobajima \cite{IkedaSobajima2} 
a blowup result for $1<p\leq p_S(N+\mu)$ is given 
for small damping case $\mu\in (0,\frac{N^2+N+2}{N+2})$, where 
$p_S(n)$ is well-known Strauss exponent given by the positive root of the quadratic equation
$(n-1)p^2-(n+1)p-2=0$.
We also refer the reader to
D'Abbicco \cite{DAbbicco15} and D'Abbicco--Lucente--Reissig \cite{DLR15}
for global existence results and determination of the critical exponent for the special case $\mu=2$,
respectively.
In the scattering case $b(t)=(1+t)^{-\beta}$ with $\beta>1$, 
Lai--Takamura \cite{LTarxiv} proved the blowup result for $1<p<p_S(N)$, 
and therefore, in this case the damping term can be ignored.

On the other hand, if $b(t)=(1+t)^{-\beta}$ with $\beta<-1$, then 
the situation is completely different. 
In this case, according to the result by Ikeda--Wakasugi \cite{IWoverdamping}, the critical exponent disappears, that is, 
there exists small global solution of \eqref{ndw} for every $p>1$. 
 This phenomenon is so-called overdamping. 
This means that the case $\beta=-1$ can be regarded as 
the threshold for dividing effective and overdamping cases
which is also discussed by Wirth \cite{Wi04, Wi06, Wi07JDE, Wi07ADE} 
for the linear equation.

Recently, Ikeda--Inui \cite{IkeInpre} 
gave the blowup result 
for the case $b(t)=(1+t)^{-\beta}$, $\beta\in [-1,1)$
with the critical nonlinearity  $p=1+2/N$
together with sharp lifespan estimates as follows
\[
\lifespan(u)\sim 
\begin{cases}
\exp(C\ep^{-(p-1)})
&
{\rm if\ }b(t)=(1+t)^{-\beta},\ \beta\in (-1,1),\ p=1+\frac{2}{N},
\\
\exp\exp(C\ep^{-(p-1)})
&
{\rm if\ }b(t)=1+t,\ \beta=-1,\ p=1+\frac{2}{N}.
\end{cases}
\]

The first purpose of the present paper 
is to determine the critical exponent dividing the global existence and blowup of small solutions
to \eqref{ndw} in more general damping coefficients including 
\[
b(t)=1+t,\quad b(t)=(1+t)(1+\log(1+t)).
\]
The second is to give a sharp estimate for lifespan of blowup solutions to 
\eqref{ndw} in view of the small parameter $\ep>0$. 

Our main result for the (implicit) upper bound of
lifespan is as follows.  

\begin{theorem}\label{main}
Assume that 
\begin{gather}
\label{b.positive}
b(t)>0, \quad t\geq 0,
\\
\label{b'.limit}
b_0:=\limsup_{t\to\infty}
\Big(\frac{|b'(s)|}{b(s)^2}\Big)<1
\end{gather}
and 
$(f, g)\in (H^2(\R^N)\cap L^1(\R^N))\times (H^1(\R^N)\cap L^1(\R^N))$ 
with 
\[
\int_{\R^N}f(x)\,dx
+ 
B_0\int_{\R^N}g(x)\,dx
>0,
\] 
where 
\[
B_0:=\int_{0}^\infty \exp\Big(-\int_0^sb(\sigma)\,d\sigma\Big)\,ds<\infty
\] 
(which is valid under \eqref{b.positive} and \eqref{b'.limit}, see Lemma \ref{lem_b_properties}).
If $1<p\leq 1+\frac{2}{N}$, then 
there exists a positive constant $C>0$ such that 
the solution $u_\ep$ of \eqref{ndw} satisfies 
\[
B(\lifespan(u_\ep))
\leq 
\begin{cases}
C\ep^{-(\frac{1}{p-1}-\frac{N}{2})^{-1}}
&\text{if }1<p<1+\frac{2}{N},
\\
\exp(C\ep^{-(p-1)})
&\text{if }p=1+\frac{2}{N}, 
\end{cases}
\]
where
\begin{align}
\label{def_B}
B(t):=\int_{0}^{t}\frac{1}{b(\sigma)}\,d\sigma.
\end{align}
Additionally, if 
\begin{align}
\label{notover}
\frac{1}{b(t)}\notin L^1(0,\infty), 
\end{align}
then small data blowup phenomena occurs. 
\end{theorem}

The proof of Theorem \ref{main} is 
done by using a test function method with a solution of the conjugate
equation $\pa_t^2\widetilde{\Phi}-\Delta \widetilde{\Phi}-\pa_t(b(t)\widetilde{\Phi})=0$ 
and rescaled cut-off functions. 
Also we use the idea for deriving upper bound of lifespan in \cite{IkedaSobajima3}. 

\begin{example}\label{exam1}
\begin{itemize}
\item[(1)] 
The typical cases $b(t)=(1+t)^{-\beta}$
satisfy the condition \eqref{b'.limit}
if $\beta\in (-\infty,1)$.
In particular, we have
$b_0=0$ with 
$\frac{|b'(s)|}{b(s)^2}=\beta(1+t)^{\beta-1}$.
Also, $b(t)=(1+t)^{-\beta}$ satisfies the condition \eqref{notover} if $\beta \in [-1,\infty)$.
\item[(2)] In the scale invariant case $b(t)=\mu(1+t)^{-1}>0$ with $\mu>1$, 
we see that
$b_0=\mu^{-1} <1$
and \eqref{b'.limit} holds.
\end{itemize}
\end{example}

It is worth noticing that the lifespan estimate in Theorem \ref{main} 
is true even if we consider the following parabolic problem 
non-trivial initial data:
\begin{equation}\label{nh}
\begin{cases}
b(t)\pa_t u(x,t)-\Delta u(x,t)=u(x,t)^p, 
& (x,t)\in \R^N \times (0,T),
\\
u(x,0)=\ep f(x)\geq 0, 
& x\in \R^N.
\end{cases}
\end{equation}
This is clear if we consider the Fujita type equation 
with change of variables $u(x,t)=v(x,B(t))$
\begin{equation}\label{nhF}
\begin{cases}
\pa_s v(x,s)-\Delta v(x,s) =v(x,s)^p, 
& (x,s)\in \R^N \times (0,S),
\\
v(x,0)=\ep f(x)\geq 0, 
& x\in \R^N.
\end{cases}
\end{equation}
To selfcontainedness, we would give a short proof of lower lifespan estimates in Appendix. 


Next, we study the lower bound of lifespan of solutions to \eqref{ndw}.
In the following, we assume
\eqref{b.positive}, \eqref{notover}, and
the following stronger version of \eqref{b'.limit}:
\begin{align}%
\label{5.b'.limit}
	&\mbox{There exist}\ \gamma >0 \ \mbox{and}\ C>0 \ \mbox{such that}
	\quad
	\frac{ |b'(t)|}{b(t)^2} \le C (t+1)^{-\gamma}
	\quad \mbox{for}\ t > 0.
\end{align}%
We denote by
$H^{s,m}(\mathbb{R}^N)$
with
$s \in \mathbb{Z}_{\ge 0}$ and $m \ge 0$
the weighted Sobolev space
\begin{align*}%
	H^{s,m}(\mathbb{R}^N) = \left\{
		f \in L^2(\mathbb{R}^N) ;
			\| f \|_{H^{s,m}} = \sum_{|\alpha| \le s} \|(1+|x|)^m \partial_x^{\alpha}f \|_{L^2}
			< \infty \right\}.
\end{align*}%
We consider the initial data belonging to
\begin{align}%
\label{5.ini}
	(f, g ) \in H^{2,m}(\mathbb{R}^N) \times H^{1,m}(\mathbb{R}^N),
\end{align}%
where
$m$
satisfies
\begin{align}%
\label{5.m}
	m = 1 \ (N=1), \quad m > \frac{N}{2} \ (N\ge 2).
\end{align}%
We remark that
\eqref{5.ini}--\eqref{5.m} imply
$(f, g) \in (H^2(\mathbb{R}^N)\cap L^1(\mathbb{R}^N))
\times  (H^1(\mathbb{R}^N)\cap L^1(\mathbb{R}^N))$.
For the nonlinearity, we assume
\begin{align}%
\label{5.p}
	1<p<\infty \ (N = 1,2), \quad 1< p \le \frac{N}{N-2} \ (N\ge 3).
\end{align}%
Under these assumptions, the existence of a unique local solution
\begin{align*}%
	u_{\varepsilon} \in C([0,T) ; H^{2,m}(\mathbb{R}^N)) \cap
		C^1([0,T) ; H^{1,m}(\mathbb{R}^N))
\end{align*}%
to \eqref{ndw}
has already been proved by \cite[Propositions 3.5, 3.6]{Wakasugi17}.
Thus, we define
$\lifespan (u_{\varepsilon})$
by the maximal existence time of the local solution.
In this section, following the argument in
\cite{FujiwaraIkedaWakasugi, Wakasugi17},
we prove the sharp lower bound of the lifespan.

\begin{theorem}\label{dw.below}
Assume that
\eqref{b.positive}, \eqref{notover}, \eqref{5.b'.limit}, and \eqref{5.ini}--\eqref{5.p}
are satisfied.
Then, there exist
$\varepsilon_0 > 0$
and
$C>0$
such that for every
$\varepsilon \in (0, \varepsilon_0]$,
one has
\begin{align}%
\label{5.dw.below}
	B(\lifespan(u_\ep))
	\geq 
		\begin{cases}
			C\ep^{-(\frac{1}{p-1}-\frac{N}{2})^{-1}}
				&\text{if }1<p<1+\frac{2}{N},
				\\
			\exp(C\ep^{-(p-1)})
				&\text{if }p=1+\frac{2}{N}, 
				\\
			\infty
				&\text{if }p>1+\frac{2}{N}.
		\end{cases}
\end{align}%
\end{theorem}

The proof is based on the method of scaling variables
introduced by
Gallay and Raugel \cite{GaRa98}.
In \cite{Wakasugi17, FujiwaraIkedaWakasugi},
the global existence, asymptotic behavior,
and the lower bound of the lifespan of solutions to
the nonlinear problem \eqref{ndw} with
$b(t) = (1+t)^{-\beta} \ (-1\le \beta<1)$ are studied,
which is a typical example satisfying
\eqref{b.positive}, \eqref{notover} and \eqref{5.b'.limit}.

\begin{remark}
From the proof of the supercritical case $p > 1+\frac{2}{N}$ with small additional argument,
we can also have the asymptotic profile of the global solution $u_{\ep}$.
More precisely, we can prove
\begin{align*}
	\| u_{\ep}(\cdot,t) - \alpha^{\ast} (4\pi B(t))^{-\frac{N}{2}} e^{-\frac{|x|^2}{4B(t)}} \|_{L^2}
	=  O(B(t)^{-\frac{N}{4}-\lambda})
\end{align*}
with some $\lambda > 0$
as $t\to \infty$,
where $\alpha^{\ast} = \lim_{t\to \infty}\int_{\mathbb{R}^N}u_{\ep}(x,t)\,dx$.
For the detail, see \cite[Section 3.9]{Wakasugi17}.
\end{remark}

Here, we note that
$b(t) = \mu(1+t)^{-1}$
with $\mu > 0$ does not satisfy \eqref{5.b'.limit} (see Example \ref{exam1} (2) above).
Thus, this case is excluded for the lower bound of the lifespan. 

To illustrate the result of the present paper, 
we give several lifespan estimates for the typical damping coefficients.
\begin{corollary}\label{cor}
Under the assumptions in Theorems \ref{main} and 
\ref{dw.below},
for respective cases, 
one has the following explicit bound of $\lifespan(u_\ep)$:
\begin{itemize}
\item[\bf (i)] If $b(t)=(1+t)^{-\beta}$ with $\beta \in (-1,1)$,
then 
\[
\lifespan(u_\ep)
\sim  
\begin{cases}
C\ep^{-(1+\beta)^{-1}(\frac{1}{p-1}-\frac{N}{2})^{-1}}
&\text{if }1<p<1+\frac{2}{N},
\\
\exp(C\ep^{-(p-1)})
&\text{if }p=1+\frac{2}{N}.
\end{cases}
\]
\item[\bf (ii)] If $b(t)=1+t$, then 
\[
\lifespan(u_\ep)
\sim 
\begin{cases}
\exp(C\ep^{-(\frac{1}{p-1}-\frac{N}{2})^{-1}})
&\text{if }1<p<1+\frac{2}{N},
\\
\exp(\exp(C\ep^{-(p-1)}))
&\text{if }p=1+\frac{2}{N}.
\end{cases}
\]
\item[\bf (iii)] If $b(t)=(1+t)(1+\log (1+t))$, then 
\[
\lifespan(u_\ep)
\sim 
\begin{cases}
\exp(\exp(C\ep^{-(\frac{1}{p-1}-\frac{N}{2})^{-1}}))
&\text{if }1<p<1+\frac{2}{N},
\\
\exp(\exp(\exp(C\ep^{-(p-1)})))
&\text{if }p=1+\frac{2}{N}.
\end{cases}
\]
\item[\bf (iv)] If $b(t)=\prod_{k=1}^n\ell_k(t)$ $(n\geq 3)$ with 
$\ell_1(t)=1+t$ and  $\ell_{k+1}(t)=1+\log (\ell_k(t))$, then 
\[
\lifespan(u_\ep)
\sim 
\begin{cases}
\exp^{[n]}(C\ep^{-(\frac{1}{p-1}-\frac{N}{2})^{-1}})
&\text{if }1<p<1+\frac{2}{N},
\\
\exp^{[n+1]}(C\ep^{-(p-1)})
&\text{if }p=1+\frac{2}{N}
\end{cases}
\]
where 
$\exp^{[1]}(t)=\exp(t)$ and $\exp^{[k+1]}(t)=\exp(\exp^{[k]}(t))$.
\item[\bf (v)]
If $b(t)=\mu(1+t)^{-1}$ with
$\mu > 1$,
then 
\[
\lifespan(u_\ep)
\le  
\begin{cases}
C\ep^{-(\frac{2}{p-1}-N)^{-1}}
&\text{if }1<p<1+\frac{2}{N},
\\
\exp(C\ep^{-(p-1)})
&\text{if }p=1+\frac{2}{N}.
\end{cases}
\]
\end{itemize}
\end{corollary}
\begin{remark}\label{rem_cor}
In the case $b(t)=\mu(1+t)^{-1}$ with $\mu > 1$, 
we do not obtain any lower bound for lifespan.
The difficulty comes from the scale-invariant property of 
the damping term which breaks the advantage of the method of 
scaling variables in the proof of Theorem \ref{dw.below}.
Moreover, 
from another point of view, 
the upper lifespan estimate of solutions with scale-invariant damping
has a wave-like profile as in \cite{IkedaSobajima2} at least when $\mu\in (0,\frac{N^2+N+2}{N+2})$. 
If $\mu\geq \frac{N^2+N+2}{N+2}$, 
then we do not know whether the upper bound of lifespan in Corollary \ref{cor} {(v)} is sharp or not. 
\end{remark}

This paper is organized as follows: 
In Section 2, we collect important properties of the damping coefficient $b(t)$ 
and the profile of the solution to the linear conjugate equation of \eqref{ndw}. 
Section 3 is devoted to prove upper bound for lifespan of solutions to 
\eqref{ndw} via a (time-rescaled) test function method 
with the solution of the linear conjugate equation. 
To close this paper, we give a proof of lower bound of lifespan of solutions to 
\eqref{ndw} via the method of scaling variables.

\section{Preliminaries}

Here we collect some basic properties of the damping coefficient $b(t)$ 
and the behavior of solution to 
\[
\pa_t^2\Phi-\Delta \Phi -\pa_t(b(t)\Phi)=0.
\] 
At the end, we introduce a family of cut-off functions with time-rescaling 
$s \sim \int_0^tb(\sigma)^{-1}\,d\sigma$. 

\subsection{Basic properties of the damping coefficient $b(t)$}

First we prove some basic properties of $b(t)$,  which we frequently use later. 
\begin{lemma}\label{lem_b_properties}
Assume that \eqref{b.positive} and \eqref{b'.limit} are satisfied. 
Then one has
\begin{gather}
\label{b.integrable}
B_0:=\int_{0}^\infty \exp\Big(-\int_0^sb(\sigma)\,d\sigma\Big)\,ds<\infty,
\\
\label{b.limit}
\lim_{t\to\infty}
\left(\frac{1}{b(t)}\exp\Big(-\int_0^tb(\sigma)\,d\sigma\Big)\right)=0.
\end{gather}
Assume further that \eqref{notover} and \eqref{5.b'.limit} are satisfied. Then 
\begin{gather}
\label{5.b.limit}
	\lim_{t\to \infty} \frac{1}{b(t)^2 (B(t)+1)} = 0.
\end{gather}
\end{lemma}
\begin{proof}
By \eqref{b'.limit}, there exist $t_0>0$ and $\delta>0$ such that for every $t\geq t_0$, 
\[
\left|\frac{d}{dt}b(t)^{-1}\right|\leq 1-\delta.
\]
This yields that for every $t\geq t_1=\max\{t_0,2[\delta b(t_0)]^{-1}\}$,
\[
b(t)^{-1}\leq b(t_0)^{-1}+(1-\delta)(t-t_0)\leq (1-\delta/2)t.
\]
Therefore
\[
\exp\left(-\int_{0}^tb(\sigma)\,d\sigma\right)
\leq 
\exp\left(-\int_{0}^{t_1}b(\sigma)\,d\sigma\right)\times 
t_1^{(1-\delta/2)^{-1}}
t^{-(1-\delta/2)^{-1}}.
\]
This implies \eqref{b.integrable} and also \eqref{b.limit}. 

On the other hand, assume \eqref{notover} and \eqref{5.b'.limit}. 
Taking $\tau>1$ arbitrary, we see from \eqref{5.b'.limit} that 
\begin{align*}
\widetilde{B}(t)=\frac{1}{b(t)^2 (B(t)+1)}
&=
\frac{b(\tau)^{-2}+2\int_\tau^tb(\sigma)^{-1}
\frac{d}{dt}[b(\sigma)^{-1}]\,d\sigma}{B(t)+1}
\\
&\leq 
\frac{b(\tau)^{-2}}{B(t)+1}
+
2C(1+\tau)^{-\gamma}\frac{B(t)}{B(t)+1}. 
\end{align*}
Using \eqref{notover}, we deduce $\limsup_{t\to\infty}\widetilde{B}(t)\leq 2C(1+\tau)^{-\gamma}$ 
and then \eqref{5.b.limit} is shown. 
\end{proof}

\subsection{Construction of solutions to the conjugate equation}

To find blowup phenomena, we will 
use the solution of the conjugate linear equation of \eqref{ndw}
\[
\pa_t^2\Phi(x,t)-\Delta \Phi(x,t)-\pa_t(b(t)\Phi(x,t))=0.
\]
In the current case, we can choose 
$\Phi(x,t)=\Phi(t)$ (independent of $x$). 
The equation is reduced to 
\[
\pa_t\Big(\pa_t\Phi(t)-b(t)\Phi(t)\Big)=0.
\]
The all solutions of the above equation are given by 
\[
\Phi_{c_0,c_1}(t)=
\exp\Big(\int_0^tb(\sigma)\,d\sigma\Big)
\left[
c_0+
c_1\int_{0}^t
\exp\Big(-\int_0^sb(\sigma)\,d\sigma\Big)\,ds
\right]
\]
with $c_0,c_1\in \R$. 
Then we fix the parameters $c_0,c_1\in \R$ (in the former case)
and collect the properties of 
$\Phi$, which we use later.
\begin{lemma}\label{harmonic}
Assume that \eqref{b.positive} and \eqref{b'.limit} are satisfied. 
Define for $t\geq 0$, 
\[
\Phi(t)=
\int_{t}^\infty
\exp\Big(-\int_t^sb(\sigma)\,d\sigma\Big)\,ds.
\]
Then $\Phi$ is well-defined and satisfies the following properties:
\begin{itemize}
\item[\bf (i)] $\Phi(0)=B_0$, and $\pa_t\Phi(t)-b(t)\Phi(t)=-1$ for every $t\geq 0$.
\item[\bf (ii)] There exist constants 
$t_0>0$, $B_1>0$ and $B_2>0$ such that 
for every $t\geq t_0$, 
\[
\frac{B_1}{b(t)}
\leq 
\Phi(t)
\leq 
\frac{B_2}{b(t)}
\]
\item[\bf (iii)] 
For every $t\geq t_0$, 
\[
|\pa_t\Phi(t)|
\leq 
\frac{1+b_0}{1-b_0}
\]
and in particular $\pa_t\Phi$ is bounded in $[0,\infty)$.
\end{itemize}
\end{lemma}

\begin{proof}
First, in view of \eqref{b.integrable} in Lemma \ref{lem_b_properties}, we can choose
\[
c_0=B_0=\int_{0}^\infty \exp\Big(-\int_0^sb(\sigma)\,d\sigma\Big)\,ds>0, 
\quad
c_1=-1
\]
and then $\Phi_{c_0,c_1}$ is nothing but the function $\Phi$ in this lemma.
The assertion {\bf (i)} is clear by the construction of $\Phi_{c_0,c_1}$.
For {\bf (ii)}, by integration by parts and \eqref{b.limit} in Lemma \ref{lem_b_properties} we have 
\begin{align*}
\Phi(t)
&=
\int_{t}^\infty
\frac{b(s)}{b(s)}\exp\Big(-\int_t^sb(\sigma)\,d\sigma\Big)\,ds
\\
&=
\left[
-\frac{1}{b(s)}\exp\Big(-\int_t^sb(\sigma)\,d\sigma\Big)
\right]_{s=t}^{s=\infty}
-
\int_{t}^\infty
\frac{b'(s)}{b^2(s)}\exp\Big(-\int_t^sb(\sigma)\,d\sigma\Big)\,ds
\\
&=
\frac{1}{b(t)}
-
\int_{t}^\infty
\frac{b'(s)}{b^2(s)}\exp\Big(-\int_t^sb(\sigma)\,d\sigma\Big)\,ds.
\end{align*}
This implies that 
\begin{align*}
\left|
\Phi(t)-\frac{1}{b(t)}
\right|
&\leq 
\int_{t}^\infty
\frac{|b'(s)|}{b^2(s)}\exp\Big(-\int_t^sb(\sigma)\,d\sigma\Big)\,ds
\\
&\leq 
\sup_{s\geq t}\left(\frac{|b'(s)|}{b^2(s)}\right)
\int_{t}^\infty
\exp\Big(-\int_t^sb(\sigma)\,d\sigma\Big)\,ds
\\
&\leq 
\sup_{s\geq t}\left(\frac{|b'(s)|}{b^2(s)}\right)
\Phi(t).
\end{align*}
Since \eqref{b'.limit} gives that there exists $t_0>0$ 
such that $\sup_{s\geq t_0}\left(\frac{|b'(s)|}{b^2(s)}\right)\leq \frac{1+b_0}{2}<1$, we deduce 
\[
\frac{2}{3+b_0}\,\frac{1}{b(t)}\leq \Phi(t)\leq \frac{2}{1-b_0}\,\frac{1}{b(t)}.
\]
Moreover, noting that 
\[
|\pa_t\Phi(t)|=|b(t)\Phi(t)-1|
\le
\sup_{s\geq t}\left(\frac{|b'(s)|}{b^2(s)}\right)
b(t)\Phi(t)
\leq 
\frac{1+b_0}{1-b_0},
\] 
we have {\bf (iii)}.
\end{proof}
\subsection{Choice of cut-off functions}

The choice of the cut-off functions are based on that 
in Ikeda--Sobajima \cite{IkedaSobajima3}
with time rescaling.

Now we set two kinds of functions $\eta\in C^2([0,\infty))$ 
and $\eta^*\in L^\infty((0,\infty))$ 
as follows:
\[
\eta(s)
\begin{cases}
=1& \text{if}\ s\in [0, 1/2],
\\
\text{is decreasing}& \text{if}\ s\in (1/2,1),
\\
=0 & \text{if}\ s\notin [1,\infty), 
\end{cases}
\quad
\eta^*(s)
=
\begin{cases}
0& \text{if}\ s\in [0, 1/2),
\\
\eta(s)& \text{if}\ s\in [1/2,\infty).
\end{cases}
\]
\begin{definition}\label{psi}
For $p>1$, we define for $R>0$, 
\begin{align*}
\psi_R(x,t)
&=
[\eta(s_R(x,t))]^{2p'}, 
\quad 
(x,t)\in \R^N\times [0,\infty),
\\
\psi_R^*(x,t)
&=
[\eta^*(s_R(x,t))]^{2p'}, 
\quad 
(x,t)\in \R^N\times [0,\infty)
\end{align*}
with 
\[
s_R(x,t)=R^{-1}
\left(
1+|x|^{2}+
\int_{0}^t\Phi(\sigma)\,d\sigma\right).
\]
We also set
\[
P(R)=\left\{(x,t)\in \R^N\times [0,\infty)\;;\;
1+|x|^{2}+
\int_{0}^t\Phi(\sigma)\,d\sigma\leq R\right\}
\]
and $t_R>0$ as 
\[
1+\int_{0}^{t_R}\Phi(\sigma)\,d\sigma=R.
\]
\end{definition}

To deduce the upper bound for the solution to \eqref{ndw}, 
we need the following lemma 
which is essentially given by \cite{IkedaSobajima3}. 
We will only give a crucial idea of its proof. 
\begin{lemma}\label{key}
Let $\delta>0$, $C_0>0$, $R_1>0$, $\theta\geq 0$ 
and $0\leq w\in L^1_{\rm loc}([0,T);L^1(\R^N))$.
Assume that 
\[
\widetilde{R}(T):=1+\int_{0}^{T}\Phi(\sigma)\,d\sigma>R_1\]
and 
for every $R\in [R_1,\widetilde{R}(T))$, 
\begin{align}\label{criterion}
\delta 
+
\iint_{P(R)}
  w(x,t)\psi_R(x,t)
\,dx\,dt
\leq 
C_0R^{-\frac{\theta}{p'}}
\left(
\iint_{P(R)}
  w(x,t)\psi_R^*(x,t)
\,dx\,dt
\right)^{\frac{1}{p}}.
\end{align}
Then $T$ has an (implicit) upper bound as follows:
\begin{align*}
\widetilde{R}(T)
\leq 
\begin{cases}
\left(R_1^{(p-1)\theta}+(\log 2)C_0^p\theta \delta^{-(p-1)}\right)^{\frac{1}{(p-1)\theta}}
&\text{if}\ \theta>0,
\\[5pt]
\exp\left(\log R_1+(\log 2)(p-1)^{-1}C_0^p\delta^{-(p-1)}\right)
&\text{if}\ \theta=0.
\end{cases}
\end{align*}
\end{lemma}
\begin{proof}
Set 
\[
y(r)=\iint_{P(r)}
  w(x,t)\psi_r^*(x,t)
\,dx\,dt, 
\quad 
Y(R)=\int_0^R\frac{y(r)}{r}\,dr. 
\]
Then by \eqref{criterion}, we can deduce 
\[
\left(\delta+(\log 2)Y(R)\right)^p
\leq 
C_0^pR^{1-\theta(p-1)}
Y'(R).
\]
This gives the desired upper bound for $R$ and also for $1+\int_{0}^{T}\Phi(\sigma)\,d\sigma$. 
\end{proof}
\begin{lemma}\label{lem_psi_R}
Let $\psi_R$ and $\psi_R^*$ be as in Definition \ref{psi}.
Then $\psi_R$ and $\psi_R^*$ satisfy the following properties:
\begin{itemize}
\item[\bf (i)] If $(x,t)\in P(R/2)$, then $\psi_R(x,t)=1$, 
and if $(x,t)\notin P(R)$, then $\psi_R(x,t)=0$.
\item[\bf (ii)] 
There exists a positive constant $C_1$ such that 
for every $(x,t)\in P(R)$, 
\begin{align*}
|\pa_t \psi_R(x,t)|\leq C_1R^{-1}\Phi(t)[\psi_R^*(x,t)]^{\frac{1}{p}}.
\end{align*}
\item[\bf (iii)] 
There exists a positive constant $C_2$ such that 
for every $(x,t)\in P(R)$, 
\begin{align*}
|\Delta \psi_R(x,t)|\leq C_2R^{-1}
[\psi_R^*(x,t)]^{\frac{1}{p}}.
\end{align*}
\item[\bf (iv)] Further assume that \eqref{notover}. Then 
there exists a positive constant $C_3$ such that 
for every $(x,t)\in P(R)$, 
\begin{align*}
|\pa_t^2 \psi_R(x,t)|
\leq 
C_3R^{-1}[\psi_R^*(x,t)]^{\frac{1}{p}}.
\end{align*}
\end{itemize}
\end{lemma}

\begin{proof}
The assertion {\bf (i)} is trivial by the definition. 
On the other hand, 
{\bf (ii)} and {\bf (iii)} follow from standard calculations:
\begin{align*}
|\pa_t\psi_R|
&=
2p'
[\eta^*(s_R)]^{2p'-1}|\eta'(s_R)\pa_ts_R|
\\
&\leq 
2p'\|\eta\eta'\|_{L^\infty}
[\eta^*(s_R)]^{2p'-2}\frac{\Phi(t)}{R}
\\
&\leq 
\frac{2p'\|\eta\eta'\|_{L^\infty}}{R}
\Phi(t)[\psi_R^*]^{\frac{1}{p}}
\end{align*}
and 
\begin{align*}
|\Delta \psi_R|
&=
2p'\left|
(2p'-1)\eta'(s_R)^2|\nabla s_R|^2
+\eta(s_R)\eta''(s_R)|\nabla s_R|^2
+\eta(s_R)\eta'(s_R)\Delta s_R
\right|[\psi_R^*]^{\frac{1}{p}}
\\
&\leq 
2p'\left(
\frac{4(2p'-1)\|\eta'\|_{L^\infty}^2+4\|\eta\eta''\|_{L^\infty}}{R^2}|x|^2
+\frac{2N\|\eta\eta'\|_{L^\infty}^2}{R}
\right)[\psi_R^*]^{\frac{1}{p}}
\end{align*}
with $|x|^2\leq R$ on ${\rm supp}\,\psi_R$. For {\bf (iv)}, 
we see that 
\begin{align*}
|\pa_t^2\psi_R|
&=
2p'
\left|
(2p'-1)
(\eta'(s_R)\pa_ts_R)^2
+
\eta(s_R)\eta''(s_R)(\pa_ts_R)^2
+
\eta(s_R)\eta'(s_R)\pa_t^2s_R
\right|
[\psi_R^*]^{\frac{1}{p}}
\\
&\leq 
2p'
\left(
(2p'-1)
\|\eta'\|_{L^\infty}^2
+
\|\eta\eta''\|_{L^\infty}
\right)
\frac{\Phi(t)^2}{R^2}[\psi_R^*]^{\frac{1}{p}}
+
2p'\|\eta\eta'\|_{L^\infty}\frac{\Phi'(t)}{R}
[\psi_R^*]^{\frac{1}{p}}.
\end{align*}
Here using $1+ \int_{0}^t\Phi(\sigma)\,d\sigma\leq R$ and 
Lemma \ref{harmonic} {\bf (iii)}, we have 
\begin{align*}
\frac{\Phi(t)^2}{R}
\leq
\frac{B_0^2+2\int_{0}^t\Phi(\sigma)\Phi'(\sigma)\,d\sigma}
{1+\int_{0}^t\Phi(\sigma)\,d\sigma}
\leq
\frac{B_0^2+2\|\Phi'\|_{L^\infty}\int_{0}^t\Phi(\sigma)\,d\sigma}
{1+ \int_{0}^t\Phi(\sigma)\,d\sigma}
\leq
\max\{B_0^2, 2\|\Phi'\|_{L^\infty}\}.
\end{align*}
Hence we obtain {\bf (iii)}. 
\end{proof}

\section{Upper bound of lifespan}
In this section we prove Theorem \ref{main}. 
\begin{proof}[Proof of Theorem \ref{main}]
Let $u$ be a solution to \eqref{ndw} in $[0,T)$ with $T=\lifespan(u)$. 
We assume $T> t_{R_0}$ with large $R_0$ determined later 
(otherwise the assertion is obvious).  
Multiplying the equation in \eqref{ndw} to $\Phi(t)\psi_R(x,t)$ 
and using integration by parts, we have 
\begin{align}
\nonumber
\int_{\R^N}|u|^p\Phi(t)\psi_R\,dx
&=
\int_{\R^N}\Big(\pa_t^2u-\Delta u + b(t)\pa_t u\Big)\Phi(t)\psi_R\,dx
\\
\nonumber
&=
\frac{d}{dt}
\left(
\int_{\R^N}\pa_tu\Phi(t)\psi_R-u\pa_t(\Phi(t)\psi_R)
+
b(t)u\Phi(t)\psi_R
\,dx
\right)
\\
\label{eq:1st}
&\quad 
+\int_{\R^N}
   u\Big(\pa_t^2(\Phi(t)\psi_R)
   -\Delta (\Phi(t)\psi_R)
   -\pa_t(b(t)\Phi(t)\psi_R)
   \Big)
\,dx.
\end{align}
It follows from Lemmas \ref{harmonic} and \ref{lem_psi_R} that
\begin{align*}
&|
\pa_t^2(\Phi(t)\psi_R)
   -\Delta (\Phi(t)\psi_R)
   -\pa_t(b(t)\Phi(t)\psi_R)
|
\\
&\leq 
2 |\Phi'(t)\pa_t\psi_R|
+
\Phi(t)|\pa_t^2\psi_R|
+
\Phi(t)|\Delta\psi_R|
+
b(t)\Phi(t)|\pa_t\psi_R|
\\
&\leq 
\frac{2C_1}{R} |\Phi'(t)|\Phi(t)[\psi_R^*]^{\frac{1}{p}}
+
\frac{C_3}{R} \Phi(t)[\psi_R^*]^{\frac{1}{p}}
+
\frac{C_2}{R}\Phi(t)[\psi_R^*]^{\frac{1}{p}}
+
\frac{B_2C_1}{R}\Phi(t)[\psi_R^*]^{\frac{1}{p}}
\\
&\leq 
\frac{C_4}{R}
\Phi(t)[\psi_R^*]^{\frac{1}{p}}
\end{align*}
with $C_4=2C_1\|\Phi'\|_{L^{\infty}}
+
C_3
+
C_2
+
B_2C_1$. 
Therefore integrating \eqref{eq:1st} on $[0,t_R]$ 
and using the above estimate, we have 
\begin{align}
\label{eq:2nd}
j(R)\ep 
+ 
\iint_{P(R)}|u|^p\Phi(t)\psi_R\,dx\,dt
&\leq 
\frac{C_4}{R}
\int_{\R^N}
   |u|\Phi(t)[\psi_R^*]^{\frac{1}{p}}
   \,dx,
\end{align}
where 
\[
j(R)=
\int_{\R^N}
(f(x)+B_0g(x))\psi_R(x,0)
\,dx
-
B_0
\int_{\R^N}
f(x)\pa_t\psi_R(x,0)
\,dx.
\]
Noting that 
\[
\|\pa_t\psi_R(\cdot,0)\|_{L^\infty}=
2p'\sup_{x\in\R^N}
\Big(
[\eta^*(s_R(x,0))]^{2p'-1}\eta'(s_R(x,0))\frac{\Phi(0)}{R}\Big)
\leq \frac{2p'\|\eta\|_{L^\infty}\Phi(0)}{R}
\to 0
\]
as $R\to\infty$, we see from the dominated convergence theorem that 
there exists $R_0>0$ such that for every $R\geq R_0$, 
\[
j(R)\geq c_0=
\frac{1}{2}\int_{\R^N}
(f(x)+B_0g(x))
\,dx>0.
\]
Therefore by \eqref{eq:2nd} with the H\"older inequality, 
we have 
\begin{align*}
c_0\ep + 
\iint_{P(R)}|u|^p\Phi(t)\psi_R\,dx\,dt
\leq 
\frac{C_4}{R}
\left(
\iint_{P(R)}\Phi(t)dx\,dt
\right)^\frac{1}{p'}
\left(
\iint_{P(R)}|u|^p\Phi(t)\psi_R^*\,dx\,dt
\right)^\frac{1}{p}.
\end{align*}
Since 
\[
\iint_{P(R)}\Phi(t)dx\,dt
\leq 
\int_0^{t_R}\int_{B(0,\sqrt{R})}\Phi(t)\,dx\,dt
=
|S^{N-1}|R^{1+\frac{N}{2}}, 
\]
The last inequality gives 
\begin{align*}
c_0\ep + 
\iint_{P(R)}|u|^p\Phi(t)\psi_R\,dx\,dt
\leq 
C_5R^{-(\frac{1}{p-1}-\frac{N}{2})\frac{1}{p'}}
\left(
\iint_{P(R)}|u|^p\Phi(t)\psi_R^*\,dx\,dt
\right)^\frac{1}{p}.
\end{align*}
Applying Lemma \ref{key} with $w(x,t)=|u(x,t)|^p\Phi(t)$, we deduce
\[
1+\int_{0}^{T}\Phi(\sigma)\,d\sigma
\leq 
\begin{cases}
C\ep^{-(\frac{1}{p-1}-\frac{N}{2})^{-1}}
&\text{if }1<p<1+\frac{2}{N},
\\
\exp(C\ep^{-(p-1)})
&\text{if }p=1+\frac{2}{N}.
\end{cases}
\]
By Lemma \ref{harmonic} {\bf (ii)}, we obtain
\[
B(T)=\int_{0}^{T}\frac{1}{b(\sigma)}\,d\sigma
\leq 
\begin{cases}
C'\ep^{-(\frac{1}{p-1}-\frac{N}{2})^{-1}}
&\text{if }1<p<1+\frac{2}{N},
\\
\exp(C'\ep^{-(p-1)})
&\text{if }p=1+\frac{2}{N}.
\end{cases}
\]
The proof is complete.
\end{proof}

\section{Lower bound of lifespan}
In this section, we discuss the lower bound of lifespan for \eqref{ndw}.
Since the following proof is the almost same as those of
\cite{Wakasugi17, FujiwaraIkedaWakasugi},
we give only the outline of the proof of Theorem \ref{dw.below}.

In what follows, for simplicity, we denote by
$u$
the solution of \eqref{ndw} instead of
$u_{\varepsilon}$.
We first apply the changing variables
\begin{align}%
\label{5.scalingvariables}
	y= (B(t)+1)^{-1/2} x, \quad
	s= \log (B(t) + 1)
\end{align}%
to the equation \eqref{ndw}.
Conversely, we also have
$t= t(s) = B^{-1}(e^s-1)$.
If we introduce a new unknown function
$(v,w)$
by the relation
\begin{align*}%
	u(x,t) &= (B(t)+1)^{-N/2} v ((B(t)+1)^{-1/2}x,\log (B(t)+1)),\\
	u_t(x,t) &= b(t)^{-1} (B(t)+1)^{-N/2-1} w ((B(t)+1)^{-1/2}x,\log (B(t)+1)),
\end{align*}%
then we have the first order system
\begin{align}%
\label{5.eq.vw}
	\left\{ \begin{array}{ll}
	\displaystyle
		v_s - \frac{y}{2}\cdot \nabla_y v - \frac{N}{2} v = w,
			&(y,s) \in\mathbb{R}^N \times (0,S),\\[8pt]
	\displaystyle
		      \frac{e^{-s}}{b(t(s))^2}
		\left( w_s - \frac{y}{2}\cdot \nabla_y w - \left( \frac{N}{2} + 1 \right) w \right) + w
			= \Delta_y v + \frac{b'(t(s))}{b(t(s))^2} w
				+ e^{\frac{N}{2}\left( 1+\frac{2}{N} - p \right)s} |v|^p,
			&(y,s) \in\mathbb{R}^N \times (0,S),\\
	\displaystyle
		v(y,0) = v_0(y) = \varepsilon f(y),\quad
		w(y,0) = w_0(y) = \varepsilon g(y),
			&y\in \mathbb{R}^N.
	\end{array}\right.
\end{align}%
We first recall the local existence result.
\begin{proposition}[\cite{Wakasugi17}]\label{5.prop.le}
Under the assumptions
\eqref{b.positive}, \eqref{notover}, \eqref{5.b.limit}, and \eqref{5.b'.limit},
there exists
$S>0$
depending only on the norm
$\| (v_0, w_0) \|_{H^{2,m}\times H^{1,m}}$
such that the Cauchy problem \eqref{5.eq.vw} admits a unique strong solution
$(v,w)$
satisfying
\begin{align}%
\label{5.vw.reg}
	(v,w) \in C([0,S) ; H^{2,m}(\mathbb{R}^N) \times H^{1,m}(\mathbb{R}^N))
		\cap C^1([0,S) ; H^{1,m}(\mathbb{R}^N) \times H^{0,m}(\mathbb{R}^N)).
\end{align}%
Moreover, we have the almost global existence for small data, namely,
for arbitrary fixed time
$S'>0$,
by taking
$\varepsilon$
sufficiently small, we can extend the solution to the interval
$[0,S']$
with the estimate
\begin{align}%
\label{5.s'}
	\| (v, w)(S') \|_{H^{1,m} \times H^{0,m}}
		\le C \varepsilon \| (v_0, w_0) \|_{H^{1,m} \times H^{0,m}}.
\end{align}%
Finally, we have the blow-up alternative, namely, if
\begin{align*}%
	\lifespan(v,w) = \sup \left\{ S \in (0,\infty) ;
		\mbox{there exists a unique strong solution $(v,w)$ to \eqref{5.eq.vw} in $(0,S)$}  \right\}
\end{align*}%
is finite, then
$\lim_{s \to \lifespan(v,w)} \| (v, w)(s) \|_{H^{1,m}\times H^{0,m}} = \infty$
holds.
\end{proposition}
For the proof of this proposition, we first prepare the local theory
for the initial data in
$H^{1,m}(\mathbb{R}^N) \times H^{0,m}(\mathbb{R}^N)$
and then, we have the regularity of the solution \eqref{5.vw.reg}
by the property of persistence regularity.
For the detail, see \cite[Proposition 3.6]{Wakasugi17}.

\subsection{A priori estimate and the proof of Theorem 1.2}
Our first goal is to obtain the following a priori estimate for the first order energy.
For a constant
$s_0 \ge 0$,
we let for
$s \ge s_0$,
\begin{align}%
\label{5.energy.M}
	M(s) = \sup_{s_0 \le \sigma \le s}
		\left( \| v(\sigma) \|_{H^{1,m}}^2
			+ \frac{e^{-\sigma}}{b(t(\sigma))^2} \| w(\sigma) \|_{H^{0,m}}^2 \right).
\end{align}%
\begin{proposition}\label{5.apriori}
Under the assumptions
\eqref{b.positive}, \eqref{notover}, \eqref{5.b.limit}, \eqref{5.b'.limit}, and \eqref{5.ini}--\eqref{5.p}
there exist constants
$s_0 \ge 0$
and
$C>0$
such that for any
$s \ge s_0$
and for a solution
$(v,w)$
to \eqref{5.eq.vw} on the interval
$[0, s]$,
we have the a priori estimate
\begin{align}%
\label{5.M.apriori}
	M(s) \le C M(s_0) + C
	\begin{cases}
		\displaystyle
			e^{N(1+2/N-p)s}M(s)^p + e^{\frac{N}{2}\left(1+2/N-p\right)s} M(s)^{\frac{p+1}{2}}
				&\mbox{if} \ 1<p<1+\frac{2}{N},\\
		\displaystyle
			s \left( M(s)^p + M(s)^{\frac{p+1}{2}} \right)
				&\mbox{if} \ p = 1+\frac{2}{N},\\
		\displaystyle
			M(s)^p + M(s)^{\frac{p+1}{2}}
				&\mbox{if} \ p> 1+\frac{2}{N}.
	\end{cases}
\end{align}%
\end{proposition}
We will give an outline of the proof of this proposition later.
Here, we prove Theorem \ref{dw.below} from Proposition \ref{5.apriori}.

\begin{proof}[Proof of Theorem \ref{dw.below}]
Let
$s_0$
be the constant given in Proposition \ref{5.apriori}.
From Proposition \ref{5.prop.le}, we have the unique local solution
$(v,w)$
having
$s_0 < \lifespan(v,w)$,
provided that
$\varepsilon$
is sufficiently small.
Moreover, by \eqref{5.s'}, we have the estimate
\begin{align*}%
	M(s_0) \le C \varepsilon^2 \| (v_0, w_0) \|_{H^{1,m}\times H^{0,m}}^2
\end{align*}%
with some constant
$C>0$.
Therefore, by \eqref{5.M.apriori}, we obtain
\begin{align}%
\label{5.M.apriori2}
	M(s) \le C_0 \varepsilon^2 I_0
	+C_1
	\begin{cases}
		\displaystyle
			e^{N(1+2/N-p)s}M(s)^p + e^{\frac{N}{2}\left(1+2/N-p\right)s} M(s)^{\frac{p+1}{2}}
				&\mbox{if} \ 1<p<1+\frac{2}{N},\\
		\displaystyle
			s \left( M(s)^p + M(s)^{\frac{p+1}{2}} \right)
				&\mbox{if} \ p = 1+\frac{2}{N},\\
		\displaystyle
			M(s)^p + M(s)^{\frac{p+1}{2}}
				&\mbox{if} \ p>1+\frac{2}{N},
	\end{cases}
\end{align}%
where
$C_0, C_1 > 0$ are some constants and
$I_0 = \| (v_0, w_0) \|_{H^{1,m}\times H^{0,m}}^2$.

We first consider the case
$1<p<1+2/N$.
Let
$S_1=S_1(\varepsilon) \ge s_0$
be the first time such that
$M(s)$ attains the value
\begin{align}%
\label{5.MS1}
	M(S_1) = 2C_0 \varepsilon^2I_0.
\end{align}%
We note that the blow-up alternative in Proposition \ref{5.prop.le} ensures such a time
$S_1$
exists if
$\lifespan (v,w) < \infty$.
We substitute
$s= S_1$
into \eqref{5.M.apriori2} to obtain
\begin{align*}%
	C_0 \varepsilon^2 I_0
	\le 2C_1 
	\max\left\{ e^{N(1+2/N-p)S_1}M(S_1)^p,
				 e^{\frac{N}{2}\left(1+2/N-p\right)S_1} M(S_1)^{\frac{p+1}{2}} \right\}.
\end{align*}%
From this and \eqref{5.MS1}, we have
\begin{align*}%
	\ep^{-(\frac{1}{p-1}-\frac{N}{2})^{-1}}
		\le C e^{S_1}
		\le C ( B(\lifespan(u)) + 1 ),
\end{align*}%
which implies \eqref{5.dw.below} in the case
$1<p<1+2/N$.

Next, we treat the case
$p = 1+2/N$.
In this case,
we take the time
$S_1$
the same as \eqref{5.MS1} and use \eqref{5.M.apriori2} to obtain
\begin{align*}%
	C_0 \varepsilon^2 I_0 
		\le 2C S_1 \max\left\{ (\varepsilon^2 I_0)^p, (\varepsilon^2 I_0)^{\frac{p+1}{2}} \right\}
		\le C S_1 \varepsilon^{p+1},
\end{align*}%
provided that
$\varepsilon$
is sufficiently small.
Thus, we conclude
\begin{align*}%
	\varepsilon^{-(p-1)} \le CS_1 \le C \log ( B(\lifespan(u)) + 1 ),
\end{align*}%
which implies \eqref{5.dw.below} in the case
$p=1+2/N$.

Finally, we consider the case
$p > 1+2/N$.
In this case, we have
\begin{align*}%
	M(s) \le C_0 \varepsilon^2 I_0
		+ C_1 \left( M(s)^p + M(s)^{\frac{p+1}{2}} \right).
\end{align*}%
From this, we have the estimate
\begin{align*}%
	M(s) \le C \varepsilon^2
\end{align*}%
for sufficiently small
$\varepsilon$.
This and the blow-up alternative imply
$\lifespan (u) = \infty$.
\end{proof}

\subsection{Outline of the proof of a priori estimate}
We give an outline of the proof of Proposition \ref{5.apriori}.
Let
\begin{align*}%
	\alpha(s) = \int_{\mathbb{R}^N} v(y,s)\,dy
\end{align*}%
and
\begin{align*}%
	\varphi_0(y) = (4\pi)^{-N/2} \exp \left( - \frac{|y|^2}{4} \right),
	\quad \psi_0(y) = \Delta_y \varphi_0(y).
\end{align*}%
We note that
$\alpha(s)$
makes sense, since
$v(s) \in H^{2,m}(\mathbb{R}^N) \subset L^1(\mathbb{R}^N)$
by \eqref{5.m}.
We decompose
$(v,w)$
into
\begin{align*}%
	v(y,s) &= \alpha(s) \varphi_0(y) + f(y,s),\\
	w(y,s) &= \frac{d\alpha}{ds}(s) \varphi_0(y)
			+ \alpha(s) \psi_0(y) + g(y,s),
\end{align*}%
where
$(f,g)$
are expected to be remainder terms.
Noting that
$\Delta_y \varphi_0 = - \frac{y}{2} \cdot \nabla_y \varphi_0 - \frac{N}{2} \varphi_0$,
$\int_{\mathbb{R}^N} \varphi_0(y)\,dy = 1$,
and the equation \eqref{5.eq.vw},
we have
\begin{align}%
\label{5.alpha'}
	\frac{d\alpha}{ds}(s) &= \int_{\mathbb{R}^N} w(y,s)\, dy,\\
\label{5.alpha''}
	\frac{e^{-s}}{b(t(s))^2} \frac{d^2 \alpha}{ds^2} (s)
		&= \frac{e^{-s}}{b(t(s))^2} \frac{d\alpha}{ds}(s)
			- \frac{d\alpha}{ds}(s)
			+ \frac{b'(t(s))}{b(t(s))^2} \frac{d\alpha}{ds}(s)
			+ e^{\frac{N}{2}\left( 1+2/N - p \right)s} \int_{\mathbb{R}^N} |v(y,s)|^p \,dy.
\end{align}%
From this, we see that
$(f,g)$
satisfies the system
\begin{align}%
\label{5.eq.fg}
	\left\{ \begin{array}{l}
	\displaystyle
		f_s - \frac{y}{2}\cdot \nabla_y f - \frac{N}{2} f = g,\\[8pt]
	\displaystyle
		\frac{e^{-s}}{b(t(s))^2}
			\left( g_s - \frac{y}{2}\cdot \nabla_y g - \left( \frac{N}{2} + 1 \right) g \right)
				+ g = \Delta_y f + \frac{b'(t(s))}{b(t(s))^2} g + h,
	\end{array}\right.
\end{align}%
where
\begin{align*}%
	h(y,s) &= 
		\frac{e^{-s}}{b(t(s))^2}
		\left( -2 \frac{d\alpha}{ds}(s) \psi_0(y)
		+\alpha(s)
		\left(\frac{y}{2}\cdot\nabla_y\psi_0(y)
			+\left(\frac{N}{2}+1\right)\psi_0(y) \right) \right)\\
		&\quad + \frac{b'(t(s))}{b(t(s))^2} \alpha(s) \psi_0(y)
			+ e^{\frac{N}{2}\left( 1+2/N - p \right)s} |v|^p
			- e^{\frac{N}{2}\left( 1+2/N - p \right)s}
				\left( \int_{\mathbb{R}^N}|v|^p\, dy \right)
				\varphi_0(y).
\end{align*}%
Moreover, by the definition of
$(f,g)$
and the equation \eqref{5.eq.fg}, we easily obtain
\begin{align}%
\label{5.fgh0}
	\int_{\mathbb{R}^N} f(s,y)\,dy = \int_{\mathbb{R}^N} g(s,y)\,dy = \int_{\mathbb{R}^N} h(s,y)\,dy = 0.
\end{align}%
In the following, based on the property \eqref{5.fgh0},
we derive energy estimates for
$(f,g)$, $\alpha$,
and
$\frac{d\alpha}{ds}$.

\subsection{Energy estimates for $N=1$}
We introduce
\begin{align*}%
	F(y,s) = \int_{-\infty}^s f(z,s)\,dz,\quad
	G(y,s) = \int_{-\infty}^s g(z,s)\,dz,\quad
	H(y,s) = \int_{-\infty}^s h(z,s)\,dz.
\end{align*}%
Here, we note that the property \eqref{5.fgh0} implies
\begin{align*}%
	\| F(s) \|_{L^2} \le C \| y f(s) \|_{L^2},
\end{align*}%
(see \cite[Lemma 3.9]{Wakasugi17})
and the same estimates hold for
$G$ and $H$.
Moreover, from the equation \eqref{5.eq.fg}, we derive the following system for
$F$ and $G$.
\begin{align}
\label{eq_FG}
	\left\{\begin{array}{l}
	\displaystyle F_s-\frac{y}{2}F_y = G,\\[8pt]
	\displaystyle
    \frac{e^{-s}}{b(t(s))^2}\left( G_s - \frac{y}{2}G_y -G \right) + G
	= F_{yy} + \frac{b'(t(s))}{b(t(s))^2} G + H
	\end{array}\right.
\end{align}
We define
\begin{align*}
	E_0(s) &= \int_{\mathbb{R}}
		\left( \frac{1}{2}\left( F_y^2 +  \frac{e^{-s}}{b(t(s))^2}G^2 \right)
		+ \frac{1}{2}F^2 + \frac{e^{-s}}{b(t(s))^2} FG \right) dy,\\
	E_1(s) &= \int_{\mathbb{R}}
		\left( \frac{1}{2} \left( f_y^2 + \frac{e^{-s}}{b(t(s))^2}g^2 \right)
		+ f^2 + 2\frac{e^{-s}}{b(t(s))^2}fg \right)dy,\\
	E_2(s) & = \int_{\mathbb{R}} y^2
		\left[ \frac{1}{2} \left( f_y^2 + \frac{e^{-s}}{b(t(s))^2}g^2 \right)
		+ \frac{1}{2} f^2 + \frac{e^{-s}}{b(t(s))^2}fg \right] dy,\\
	E_3(s) &= \frac{1}{2} \frac{e^{-s}}{b(t(s))^2}\left( \frac{d\alpha}{ds}(s) \right)^2
		+ e^{-s/2}\alpha(s)^2,\\
	E_4(s) &= \frac{1}{2}\alpha(s)^2
		+ \frac{e^{-s}}{b(t(s))^2}\alpha(s) \frac{d\alpha}{ds}(s),
\end{align*}
and
\begin{align*}
	E_5(s) = \sum_{j=0}^4 C_j E_j(s),
\end{align*}
where
$C_j\ (j=0,\ldots, 4)$
are constants such that
$C_2 = C_3 = C_4 =1$
and
$1 \ll C_1 \ll C_0$.
By a straightforward calculation, we can see that
there exists sufficiently large
$s_0 >0$
such that
$E_5(s)$
has the bound
\begin{align*}
	E_5(s) \sim
	\| f(s) \|_{H^{1,1}}^2 + \frac{e^{-s}}{b(t(s))^2} \| g(s) \|_{H^{0,1}}^2
		+ \alpha(s)^2 +  \frac{e^{-s}}{b(t(s))^2} \left( \frac{d\alpha}{ds}(s) \right)^2
\end{align*}
for
$s \ge s_0$.
Furthermore, we have the following energy estimate.

\begin{lemma}\label{lem_en0}
{{\rm (\cite[Lemma 4.4]{FujiwaraIkedaWakasugi}, \cite[Lemmas 3.10--3.17]{Wakasugi17})}}
There exists
$s_0 >0$
such that we have the energy identity
\begin{align*}
	\frac{d}{ds}E_5(s) + \frac12 \sum_{j=0}^3 C_j E_j(s) + L_5(s) = R_5(s)
\end{align*}
for
$s >0$,
where
$L_5(s)$
has the lower bound
\begin{align*}%
	&\| f(s) \|_{H^{1,1}}^2 + \| g(s) \|_{H^{0,1}}^2 + \left( \frac{d\alpha}{ds}(s) \right)^2
	\le C L_5(s)
\end{align*}%
for
$s \ge s_0$,
and
$R_5(s)$
satisfies the estimate
\begin{align*}
	|R_5(s)| &\le \frac12 L_5(s)
		+ C e^{-\gamma s} E_5(s)
		+ C e^{(3-p)s} E_5(s)^p
		+ C e^{\frac{3-p}{2} s} E_5(s)^{\frac{p+1}{2}}
\end{align*}
for
$s \ge s_0$, 
where
$\gamma$
is given in \eqref{5.b'.limit}.
\end{lemma}%
\begin{remark}
{\rm (i)}
We can write down
$L_5(s)$
and
$R_5(s)$
explicitly (see \cite[Lemma 4.4]{FujiwaraIkedaWakasugi}, \cite[Lemmas 3.10--3.17]{Wakasugi17}).

\noindent
{\rm (ii)}
The term
$e^{-\gamma s} E_5(s)$
comes from
$b(t(s))^{-1} \frac{db}{dt}(t(s)) \alpha (s) \frac{d\alpha}{ds}(s)$,
which appears in the remainder term
$R_5(s)$,
since we have
$E_5(s) \ge \alpha(s)^2$
but we do not have
$L_5(s) \ge \alpha(s)^2$.
\end{remark}

The proof of the above lemma is
the completely same as that of
\cite[Lemmas 3.10--3.17]{Wakasugi17},
which needs simple but tedious computations.
Hence, we omit the detail.

\subsection{Energy estimates for $N \ge 2$}
When
$N \ge 2$,
we introduce
\begin{align*}
	\hat{F}(\xi,s) = |\xi|^{-N/2-\delta}\hat{f}(\xi,s),\quad
	\hat{G}(\xi,s) = |\xi|^{-N/2-\delta}\hat{g}(\xi,s),\quad
	\hat{H}(\xi,s) = |\xi|^{-N/2-\delta}\hat{h}(\xi,s),
\end{align*}
where
$0<\delta<1$,
and
$\hat{f}(\xi,s)$ denotes the Fourier transform of $f(y,s)$ with respect to
the space variable.

The following lemma is an improvement of
\cite[Lemma 3.11]{Wakasugi17}
by Fukuya \cite{Fukuya}.
\begin{lemma}\label{lem_hardy2}
Let $m>N/2$ and $f(y) \in H^{0,m}(\mathbb{R}^N)$ be a function satisfying
$\hat{f}(0) = (2\pi)^{-N/2} \int_{\mathbb{R}^N}f(y)dy = 0$.
Let
$\hat{F}(\xi) = |\xi|^{-N/2-\delta}\hat{f}(\xi)$
with some $0<\delta<1$.
Then,
there exists a constant $C(N,m,\delta)>0$ such that
\begin{align}
\label{hardy2}
	\| F \|_{L^2} \le C(N,m,\delta) \| f \|_{H^{0,m}}
\end{align}
holds.
\end{lemma}
\begin{remark}
In \cite[Lemma 3.11]{Wakasugi17}, 
the $L^\infty$-$L^1$ estimate for the Fourier 
transform is used and then an extra restriction for $m$ is needed. 
In contrast, in \cite{Fukuya} he gave a simple proof in which he used only 
the definition of Fourier transform and a basic inequality 
$|e^{-ix \cdot \xi}-1|\leq |x||\xi|$. These give 
the H\"older continuity of $\hat{f}$:
\[
|\hat{f}(\xi)-\hat{f}(0)|=
\left|\int_{\R^N}(e^{-ix\cdot \xi}-1)f(x)\,dx\right|
\leq 
2^{1-\delta'}|\xi|^{\delta'}
\int_{\R^N}|x|^{\delta'}|f(x)|\,dx, \quad (|\xi|<1, 0\le \delta' \le 1)
\]
As a result, he could prove the assertion of Lemma \ref{lem_hardy2} 
by assuming only the condition $m>N/2$ which is equivalent to the continuity of 
the inclusion $H^{0,m}\subset L^1$. 
\end{remark}
Also, by a direct calculation, one has
\begin{align*}
	\| f \|_{L^2}^2 \le \eta \| \nabla f \|_{L^2}^2 + C \| \nabla F \|_{L^2}^2
\end{align*}
for any small
$\eta>0$,
where the constant
$C>0$
depends on
$\eta$
(see \cite[(3.39)]{Wakasugi17}, \cite[Lemma 4.6]{FujiwaraIkedaWakasugi}).

In this case
$\hat{F}$ and $\hat{G}$ satisfy the following system:
\begin{align*}
	\left\{ \begin{array}{ll}
	\displaystyle \hat{F}_s + \frac{\xi}{2}\cdot \nabla_{\xi}\hat{F}
		+\frac{1}{2}\left( \frac{N}{2} + \delta \right) \hat{F} = \hat{G},
	&(\xi,s) \in \mathbb{R}^N\times (0,S),\\
	\displaystyle \frac{e^{-s}}{b(t(s))^2}\left( \hat{G}_s + \frac{\xi}{2}\cdot \nabla_{\xi} \hat{G}
		+ \frac{1}{2} \left( \frac{N}{2}+\delta-2 \right) \hat{G} \right) + \hat{G}
		= -|\xi|^2 \hat{F} + \frac{b'(t(s))}{b(t(s))^2} \hat{G}+ \hat{H},
	&(\xi,s) \in \mathbb{R}^N\times (0,S).
	\end{array} \right.
\end{align*}

We define the following energy
\begin{align*}
	E_0(s) &= {\rm Re} \int_{\mathbb{R}^N}
		\left( \frac{1}{2}
			\left( |\xi|^2 |\hat{F}|^2 + \frac{e^{-s}}{b(t(s))^2} |\hat{G}|^2 \right)
			+ \frac{1}{2} |\hat{F}|^2 + \frac{e^{-s}}{b(t(s))^2}\hat{F} \bar{\hat{G}}  \right)
			d\xi,\\
	E_1(s) &= \int_{\mathbb{R}^N}
		\left( \frac{1}{2}\left( |\nabla_y f |^2 + \frac{e^{-s}}{b(t(s))^2}g^2 \right)
		+ \left( \frac{N}{4} + 1 \right)
			\left( \frac{1}{2} f^2 + \frac{e^{-s}}{b(t(s))^2}fg \right) \right) dy,\\
	E_2(s) &= \int_{\mathbb{R}^N}
		|y|^{2m} \left[
			\frac{1}{2}\left( |\nabla_y f|^2 + \frac{e^{-s}}{b(t(s))^2}g^2 \right)
				+\frac{1}{2}f^2 + \frac{e^{-s}}{b(t(s))^2}fg \right] dy,\\
	E_3(s) &= \frac{1}{2}\frac{e^{-s}}{b(t(s))^2}\left( \frac{d\alpha}{ds}(s) \right)^2
		+ e^{-2\lambda s}\alpha(s)^2,\\
	E_4(s) &= \frac{1}{2}\alpha(s)^2
		+ \frac{e^{-s}}{b(t(s))^2}\alpha(s) \frac{d\alpha}{ds}(s),
\end{align*}
where $\lambda$ is a parameter satisfying
$0 < \lambda < \min\{ \frac12, \frac{m}{2}-\frac{N}{4} \}$.
Moreover, we define
\begin{align*}
	E_5(s) = \sum_{j=0}^4 C_j E_j(s),
\end{align*}
with positive constants
$C_j\ (j=0,\ldots, 4)$
such that
$C_2 = C_3 = C_4 =1$ and
$1 \ll C_1 \ll C_0$.
Then, as in the case of
$N=1$,
there exists sufficiently large
$s_0 > 0$
such that
$E_5(s)$
has the bound
\begin{align*}%
	E_5(s)
	\sim 
	\| f(s) \|_{H^{1,m}}^2 + \frac{e^{-s}}{b(t(s))^2} \| g(s) \|_{H^{0,m}}^2
		+ \alpha(s)^2 +  \frac{e^{-s}}{b(t(s))^2} \left( \frac{d\alpha}{ds}(s) \right)^2
\end{align*}%
for
$s \ge s_0$.
Moreover,
in the same manner as the case
$N=1$,
we have the following energy estimate.

\begin{lemma}\label{lem_en2}
{\rm (\cite[Lemma 4.7]{FujiwaraIkedaWakasugi}, \cite[Lemmas 3.12--3.17]{Wakasugi17}) }
There exist constants
$0 < \lambda < \min\{ \frac12, \frac{m}{2}-\frac{N}{4} \}$
and
$s_0 >0$
such that we have the energy identity
\begin{align*}
	\frac{d}{ds}E_5(s) + 2\lambda \sum_{j=0}^3 C_j E_j(s) + L_5(s) = R_5(s),
\end{align*}
for
$s >0$,
where
$L_5(s)$
has the lower bound
\begin{align*}%
	&\| f(s) \|_{H^{1,m}}^2 + \| g(s) \|_{H^{0,m}}^2 + \left( \frac{d\alpha}{ds}(s) \right)^2
	\le C L_5(s)
\end{align*}%
for
$s \ge s_0$,
and
$R_5(s)$
satisfies the estimate
\begin{align*}
	|R_5(s)| &\le \frac12 L_5(s)
		+ C e^{-\gamma s} E_5(s)
		+ C e^{N(1+2/N-p)s} E_5(s)^p
		+ C e^{\frac{N}{2}(1+2/N-p) s} E_5(s)^{\frac{p+1}{2}}
\end{align*}
for
$s \ge s_0$,
where
$\gamma$
is given in \eqref{5.b'.limit}.
\end{lemma}

\subsection{Proof of Proposition \ref{5.apriori}}
Let
$\varepsilon_1 > 0$
be sufficiently small so that
the local solution $(v,w)$ of \eqref{5.eq.vw} exists for 
$s > s_0$
(see Proposition \ref{5.prop.le}).
Therefore, we can apply Lemmas \ref{lem_en0}, \ref{lem_en2}.
Putting
\[
	\Lambda (s) := \exp \left( -C\int_{s_0}^s e^{- \gamma \sigma}\, d\sigma \right),
\]
which satisfies
$\Lambda(s) \sim 1$,
we have
\[
	\Lambda(s) E_5(s) \le E_5(s_0)
		+ C \int_{s_0}^s
		\left[ \Lambda(\sigma) e^{N(1+2/N-p) \sigma} E_5(\sigma)^p
			+ \Lambda(\sigma) e^{\frac{N}{2}(1+2/N-p) \sigma} E_5(\sigma)^{\frac{p+1}{2}}
						\right] \, d\sigma.
\]
Consequently,
if we define
$M(s)$
by \eqref{5.energy.M},
then by calculating the integral, we easily obtain \eqref{5.M.apriori},
which completes the proof.

\section*{Appendix: Lower bound of lifespan for the Fujita equation}

We would like to give a short proof of 
lower bound of the lifespan of solutions to 
the standard semilinear heat equation of Fujita type: 
\begin{equation}\label{nh}
\begin{cases}
\pa_t u(x,t)-\Delta u(x,t)=|u(x,t)|^p, 
& (x,t)\in \R^N \times (0,T),
\\
u(x,0)=f_\ep(x)=\ep f(x)\geq 0, 
& x\in \R^N.
\end{cases}
\end{equation}
The following argument is based on Quittner and Souplet \cite{QuittnerSouplet}.
Put $0\leq f\in L^1(\R^N)\cap L^\infty(\R^N)$ and set
\[
U(x,t)=h(t)^{-\frac{1}{p-1}}e^{t\Delta}f_\ep.
\]
Then we have 
\begin{align*}
(\pa_t-\Delta)U(x,t)
=
-\frac{1}{p-1}h'(t)h(t)^{-\frac{p}{p-1}}e^{t\Delta}f_\ep+h(t)(\pa_t -\Delta)e^{t\Delta}f_\ep
=
-\frac{1}{p-1}h'(t)h(t)^{-\frac{p}{p-1}}e^{t\Delta}f_\ep.
\end{align*}
If 
$-\frac{1}{p-1}h'(t)h(t)^{-\frac{p}{p-1}}e^{t\Delta}f_\ep
\geq 
|h(t)^{-\frac{1}{p-1}}e^{t\Delta}f_\ep|^{p}=|U(x,t)|^p$ and $h(0)\geq 1$, then $U(x,t)$ can be used as 
a super-solution of nonlinear problem \eqref{nh}. Solving 
$h'(t)=-(p-1)\|e^{t\Delta}f_\ep\|_{L^\infty}^{p-1}$ with 
$h(0)=1$, we have 
\[
h(t)=1-(p-1)\int_{0}^t\|e^{s\Delta}f_\ep\|_{L^\infty}^{p-1}\,ds, 
\quad t<
t_\ep=\sup\{t>0\;;\;h(t)>0\}. 
\]
Then we have $\lifespan(U)=t_\ep$. 
On the other hand, we have
\[
\|e^{t\Delta}f_\ep\|_{L^\infty}\leq 
\min\Big\{\|f_\ep\|_{L^\infty},C_Nt^{-\frac{N}{2}}\|f_\ep\|_{L^1}\Big\}
=\ep \min\Big\{\|f\|_{L^\infty},C_Nt^{-\frac{N}{2}}\|f\|_{L^1}\Big\}.
\]
This implies that for $t>1$, 
\begin{align*}
\int_{0}^t\|e^{s\Delta}f_\ep\|_{L^\infty}^{p-1}\,ds
&=
\int_{0}^1\|e^{s\Delta}f_\ep\|_{L^\infty}^{p-1}\,ds
+
\int_{1}^t\|e^{s\Delta}f_\ep\|_{L^\infty}^{p-1}\,ds
\\
&
\leq 
\ep^{p-1}
\left(
\|f\|_{L^\infty}^{p-1}
+
C_N^{p-1}\|f\|_{L^1}^{p-1}\int_1^ts^{-\frac{N}{2}(p-1)}\,ds
\right).
\end{align*}
Since by a continuity method, we can extend the solution $u$ until $t=\lifespan(U)$  
and then we have $\lifespan(u)\geq \lifespan(U)$. 
On the other hand, the above estimate yields
\[
\lifespan(U)
\geq 
\begin{cases}
C\ep^{-(\frac{1}{p-1}-\frac{N}{2})^{-1}}
& \text{if }1<p<1+\frac{2}{N}
\\
\exp(C\ep^{-(p-1)})
& \text{if }p=1+\frac{2}{N}.
\end{cases}
\]

\subsection*{Acknowledgements}
This work is partially supported 
by Grant-in-Aid for Young Scientists Research (B) 
No.16K17619,
Grant-in-Aid for Young Scientists Research (B) 
No.15K17571,
and
Grant-in-Aid for Young Scientists Research (B) 
No.16K17625.


\end{document}